\documentclass[aap]{amsart}
\input{mystyle.sty}

\title[Discrete and continuous time RL]{Reconciling Discrete-Time Mixed Policies and Continuous-Time Relaxed Controls in Reinforcement Learning and Stochastic Control}%
\author{Ren\'e Carmona$^{1}$}
\thanks{$^{1}$Department of Operations Research and Financial Engineering \& Program in Applied and Computational Mathematics, Princeton NJ 08544, USA, (\href{mailto:rcarmona@princeton.edu}{rcarmona@princeton.edu})}
\author{Mathieu Lauri\`ere$^{2}$}
\thanks{$^{2}$Shanghai Frontiers Science Center of Artificial Intelligence and Deep Learning; NYU-ECNU Institute of Mathematical Sciences at NYU Shanghai; NYU Shanghai, 567 West Yangsi Road, Shanghai, 200126, People’s Republic of China, (\href{mailto:mathieu.lauriere@nyu.edu}{mathieu.lauriere@nyu.edu})}
\thanks{{\bf Acknowledgements.}  The authors were partially supported through funding from AFOSR under the grant FA9550-23-1-0324. }

\date{}

\begin{document}

\begin{abstract}
Reinforcement learning (RL) is currently one of the most prominent methods for optimizing dynamical systems, with breakthrough results across various fields. The framework is based on the concept of a Markov decision process (MDP), leading to a discrete-time optimal control problem. In the RL literature, such problems are typically formulated and solved using mixed policies, from which random actions are sampled at each time step.
Recently, part of the optimal control community has begun investigating continuous-time versions of RL algorithms, replacing MDPs with continuous-time stochastic processes governed by relaxed controls, and asserting a full analogy between the two formulations. In this work, we examine the limitations of this analogy and rigorously establish a connection between the two problems in the case where only the drift term of the continuous-time model is controlled. We prove strong convergence of the RL implementation of mixed strategies as the time discretization mesh tends to zero. We also discuss the technical challenges posed by the possible presence of control in the diffusion component of the state.
\end{abstract}

\maketitle

\section{\textbf{Introduction}}

The spectacular successes of Reinforcement Learning (RL) no longer need to be touted. The litany of remarkable achievements made possible by its implementations includes the solution of problems long considered out of reach, such as the game of Go, robotics, and self-driving cars. More recently, RL has contributed to major advances in large AI systems, including chatbots like ChatGPT and DeepSeek. Mathematical models of RL are typically set in discrete time and, when based on a Markovian structure, are formulated in terms of Markov Decision Processes (MDPs), whose theoretical foundations are well established. See, for example, the classical treatises by Sutton and Barto~\cite{SuttonBarto} and Bertsekas and Shreve~\cite{BertsekasShreve}. While pure (deterministic) controls, leading to pure policies, remain the most natural tools for agents to influence the evolution of a system, the existence of optimal pure controls is not always guaranteed. This motivated the extension of the theory to include random controls, also known as \emph{mixed controls}. Allowing for randomization in the choice of actions brings two major benefits. First, it leads to a compactification of the model, which simplifies the analysis and facilitates existence proofs. One of the earliest and most famous examples is Nash’s theorem on the existence of equilibria in mixed strategies for finite games~\cite{Nash_1951}. Second, from a practical perspective, randomization enables the controller to explore parts of the state space that might otherwise be missed by deterministic or greedy strategies, thereby improving the approximation of optimal solutions. While practitioners often bypass the measure-theoretic difficulties that come with random controls, applied mathematicians can make use of real analysis to rigorously study open-loop models, even in systems involving many agents; see, for example,~\cite{MottePham,CarmonaLauriere_AAP}.

Control theory, in both its deterministic and stochastic forms, has also drawn sustained attention due to its broad range of applications in engineering and the social sciences, including aeronautics, network optimization, economics, and finance. Most of these models are formulated in continuous time and are typically analyzed using Ordinary Differential Equations (ODEs), Partial Differential Equations (PDEs), or Stochastic Differential Equations (SDEs), depending on whether randomness is present. For background on the PDE approach, see~\cite{FlemingSoner}, and for the probabilistic approach based on Backward Stochastic Differential Equations (BSDEs), see~\cite{Carmona_SIAM}. As in the discrete-time setting of RL, proving that optimal controls actually attain the value of the objective function remains a significant challenge. This led to the use of compactification methods similar to those employed in the discrete setting. In continuous time, the corresponding notions are Young measures, introduced by Young~\cite{Young} and further developed by Fleming~\cite{fleming1978generalized,FlemingNisio} for deterministic systems, and \emph{relaxed controls} in the stochastic case. For further background, see the survey by Borkar~\cite{Borkar_survey} or the monograph by Yong and Zhou~\cite{YongZhou}. The compactification of the space of relaxed controls was formalized using the notion of stable convergence, introduced by Jacod and M\'emin in~\cite{JacodMemin}.

The two streams of publications referred to above have developed largely in parallel, with little or no attempt to establish a rigorous connection between them. However, the remarkable successes of RL have generated a wave of interest among \emph{stochastic analysts} working with continuous-time models, some of whom have begun promoting the idea of \emph{continuous-time RL}. Apart from some early contributions such as~\cite{munos2000study,doya2000reinforcement}, the topic has gained momentum only in recent years; see, for example,~\cite{wang2020continuous,Wang_Z_Z_RL,JiaZhou,szpruch2024optimal,han2023choquet,basei2022logarithmic,szpruch2021exploration,kobeissi2022temporal,berthier2022non,kobeissi2022variance,treven2023efficient}, and the recent survey~\cite{hu2022recent}. Applications have emerged particularly in finance~\cite{wang2020continuous,huang2022achieving,hambly2023recent,wu2024reinforcement}, and the framework has been extended in several directions, including mean field games and mean field control; see, e.g.,~\cite{frikha2023actor,li2024policy,liang2024actor,wei2024unified,wei2025continuous}. 
In~\cite{Wang_Z_Z_RL}, the use of a large number of independent sample trajectories is argued to justify, at least informally, the use of the relaxed control formulation through the law of large numbers. The familiar form of randomization used in continuous-time stochastic control is presented there as a natural tool for exploring the state space. In a further attempt to connect the continuous-time and discrete-time approaches to randomization,~\cite{JiaZhou} invokes the construction of white noise in continuous time along with a Markovian projection argument, though the discussion remains largely intuitive. By contrast,~\cite{szpruch2024optimal} cautions against over-relying on white noise in continuous time, and instead recommends the use of mixed controls on a finite time grid. We investigate this question carefully and resolve this point in Section~\ref{se:after_sun} below. 
We do not consider here the effects of entropy regularization, despite its close connection with randomization in~\cite{Wang_Z_Z_RL,JiaZhou,szpruch2024optimal}. In a concurrent and independent line of work,~\cite{jia2025accuracy} analyzes the convergence of dynamics with discretely sampled actions to a system governed by relaxed controls, in a setting where both processes evolve in continuous time. This differs from our approach, which considers a discrete-time model with randomized actions, as is typically done in practical implementations of RL.

In this paper, our objective is to highlight the similarities and differences between the discrete-time and continuous-time approaches, supporting them with rigorous mathematical arguments and convergence proofs. We frequently refer to practical steps taken in RL applications, but our aim is not to propose new algorithms for approximating optimal solutions. We do not focus on optimization itself, but rather on investigating and comparing the foundational structures of the models. We review the historical motivations for introducing randomization in both continuous-time and discrete-time settings, and we demonstrate the risks of assuming too readily that the white-noise nature of mixed controls in discrete time can be translated directly into the use of relaxed controls in continuous time.

Several challenging open questions remain concerning the convergence of discrete-time models on a fine time grid toward their continuous-time counterparts. For instance,~\cite{tallec2019making} studies powerful machine learning algorithms currently employed in RL and argues that, in the regime of \emph{near continuous time}, such algorithms may \emph{collapse} when the time step becomes too small. Due to space constraints, we do not address these issues here.

The paper is organized as follows. Section~\ref{se:pure_controls} reviews the use of pure (i.e., deterministic) controls in the optimization of Markov Decision Processes (MDPs) and in the solution of optimal control problems for stochastic differential equations of It\^o type. Section~\ref{se:compactifications} surveys approaches to the same optimization problems when stochastic controls are allowed. In both cases, the emphasis is on the role of \emph{compactifying} the models to establish mathematical existence results ensuring that the optima are attained. However, we also point out that, in the context of RL applications, randomization in discrete-time models is often introduced with the goal of exploring the state space more effectively. 
The main objective of the paper is first addressed in Section~\ref{se:after_sun}, where we consider the following question: \emph{can the form of randomization used in RL applications provide a foundation for the relaxed control framework in continuous-time models?} In other words, can one meaningfully sample independently from a distribution in continuous time? In this section, we highlight the technical difficulties posed by the lack of measurability of continuous-time white noise processes, and we refute several common claims in the literature suggesting that these issues can be bypassed. 
Section~\ref{se:final} presents a strong approximation result showing that, for It\^o stochastic differential equations in which the volatility is not controlled, the RL-style independent sampling of actions does converge to trajectories of the continuous-time relaxed control model. Section~\ref{se:numeric} provides numerical evidence that the basic upper bound derived in Section~\ref{se:final} is likely to be sharp. It also shows that when the volatility is controlled, the convergence rate becomes \emph{sub-linear}, indicating that in such cases, the same arguments are no longer sufficient to establish convergence. Finally, Section~\ref{se:discussion} explains how the shift to a formulation based on the control of the \emph{martingale problem}, when volatility is controlled, offers a natural perspective and why this approach rules out the kind of discrete-time approximation via RL-type mixed controls that we study in this paper.

\section{\textbf{The State of Affairs with Pure Controls}}
\label{se:pure_controls}

In this section, we review the standard approaches to the solution of the dynamic optimization problems alluded to in the introduction, without any \emph{randomization} of the controls. We use the terminology \emph{pure control}  or \emph{strict control} to specify choices of actions as deterministic functions of the information given at a specific time, typically the past history of the system in the case of \emph{open loop} controls, 
or the history of the state of the system in the case of \emph{closed loop} or \emph{feedback} control, or of the current state of the system in the case of \emph{Markovian} control. 

\subsection{Optimization of MDPs with Pure Control Policies}
\label{sub:MDPs}
Consider an MDP defined by: 1) a state space $S$ which we assume to be a measurable space, for example a Polish space, 2) an action space $A$ which we assume to be a Polish space as well, 3) a (measurable) transition probability kernel $P: S \times A \to \cP(S)$, and 4) a time horizon $T$ which can be a finite integer or $\infty$. In order to emphasize the discrete nature of time, we use the letters, $m$, $n$, $\ldots$ to denote time which we want to think of as integers. When we discuss continuous time models below, we shall use the letters $s$, $t$, $\ldots$ to denote time which we shall think of as non-negative real numbers.

Given a control process $\balpha=(\alpha_n)_{0\le n\le T}$, namely a sequence of random variables with values in $A$, the time evolution of the state of the system is given recursively by the prescription that at time $n+1$, conditioned on the knowledge of the state $X_n$ and the choice of an action $\alpha_n\in A$ at time $n$, the state $X_{n+1}$ is a random element of the state space $S$ with distribution $P(X_n,\alpha_n)$. In notation, we write $X_{n+1}\sim P(X_n,\alpha_n)$. At each time $n$, we assume that the choice of action is made on the basis of the information avalable at that time, in other words, we assume that the random variable $\alpha_n$ is measurable with respect to the past and the present at that time.

\begin{remark}
    \label{re:system_function}
It is sometime convenient to define the time evolution of the state of the system via a \emph{system function} $F:S\times A\times E\to S$
and a dynamical equation
$$
    X_{n+1}=F(X_n,\alpha_n,\epsilon_{n+1})
$$
for an i.i.d. sequence $(\epsilon_n)_{n\ge 1}$, of innovation random variables taking values in a space $E$. This is the approach used in \cite{MottePham,CarmonaLauriere_AAP} to analyze the optimization of mean field MDPs with general \emph{open loop} and \emph{closed loop Markovian} control processes in the presence of common noise.
\end{remark}

Because of their ease of implementation, and because of their ability to preserve the desirable Markov property, \emph{closed loop} control processes in Markovian feedback form have been studied and implemented to a larger extent than the more general \emph{open loop} control processes whose implementation in practical applications is problematic. A closed loop Markovian control process $\balpha=(\alpha_n)_{0\le n\le T}$ is typically determined by a \emph{policy} $\bpi=(\pi_n)_{0\le n\le T}$ given by a sequence of measurable functions $\pi_n:S\to A$ so that the random action $\alpha_n$ taken at time $n$ is given by the value $\pi_n(X_n)$ of the function $\pi_n$ of the state $X_n$ at that time. 

\begin{remark}
It is important to realize that $X_{n+1}$ depends on the actual realizations of the past actions $(\alpha_m)_{m \le n}$, and not just on their distributions. This will come into play when we use \emph{mixed control policies}.    
\end{remark}

\vskip 4pt
The optimization problem for the control of an MDP is to minimize an expected aggregate cost:
$$
    \EE\Bigl[\sum_{n=0}^T \gamma^n f\bigl(X_n,\alpha_n\bigr)\Bigr],
$$
where $f:S\times A\to\RR$ is a one stage cost function, and $\gamma\in (0,1]$ is a discount factor which is systematically set to $\gamma=1$ when the time horizon $T$ is finite.

\vskip 2pt
It is often convenient to introduce the notion of \emph{value function} defined as
\begin{equation}
    \label{fo:discrete_value}
    V(m,x)=\inf_{\balpha}\EE\Bigl[\sum_{n=m}^T \gamma^{n-m} f\bigl(X_n,\alpha_n\bigr)\Bigl| X_m=x\Bigr].
\end{equation}
This definition will remain informal until we specify precisely the set of control processes $\balpha$, or policies, over which the optimization is to take place.
The values of this function for different values of the time variable $m$ satisfy a fundamental equation known as the \emph{Dynamic Programming Principle} (DPP), also known as the \emph{Bellman equation}, and in most cases, computing the value function is done by solving this equation.
Notice that, if the cost function $f$ and the dynamics are not known, it is not straightforward to compute an optimal control directly based on $V$. In RL, it is thus more usual to compute the \emph{state-action value function}, also called \emph{Q-function}. It also satisfies a form of DPP, which is the basis for the celebrated Q-learning algorithm of~\cite{watkins1992q}. We refer e.g. to~\cite{SuttonBarto} for more details. 

\subsection{Optimal Control in Continuous Time}
\label{sub:SDEs}
For the purpose of the present discussion, we consider the case of
an optimal stochastic control problem in which the time evolution of the state of the system is given by a Stochastic Differential Equation (SDE) of the It\^o type.  Our analysis can be generalized in all sorts of directions (SDE's with jumps, singular controls, stochastic games, Mean Field Games, \dots), and even their deterministic analogues which can be treated by setting the volatility to $0$, but we shall restrict ourselves to IT\^o SDEs for the sake of simplicity.

\vskip 2pt
So at least for now, we choose 1) a state space $S=\RR^d$, 2) an action space $A$ assumed to be a closed convex subset of $\RR^k$, 3) a drift vector field $b: [0,T]\times S \times A \to S$, 4) a volatility function $\sigma: [0,T]\times S \times A \to \cM(\RR^d)$, where $\cM(\RR^d)$ is the space of $d\times d$ real matrices, 5) a running cost function $f: S \times A \to \RR$, 6) a discount factor $\gamma=e^{-\beta} \in (0,1]$, and 7) a time horizon $T$ which can be finite or infinite. Let $\bW=(W_t)_{t\ge 0}$ be a $d$-dimensional Brownian motion.

We assume that the dynamics of the state are controlled in the sense that the SDE giving its evolution is of the form
\begin{equation}
\label{fo:state_sde}
    dX_t =  b(t,X_t, \alpha_t) dt + \sigma(t,X_t, \alpha_t) dW_t
\end{equation}
where the control process $\balpha=(\alpha_t)_{0\le t\le T}$ is a progressively measurable (\emph{non-anticipative}) $A$-valued stochastic process, $\alpha_t\in A$ representing the action taken at time $t$ by the agent/controller. We shall make clear later on what we mean by non-anticipative, but intuitively speaking, this means that the controller does not have a crystal ball, and needs to choose an action using only the information available at that time. 

\vskip 2pt
The choice of the control process should be made in order to minimize the expected future discounted cost: 
\begin{equation}
    \label{fo:expected_cost}
    \EE\Bigl[\int_0^T e^{-\beta t} f\bigl(t,X_t, \alpha_t \bigr)\;dt\Bigr],
\end{equation}
where as before, the time horizon $T>0$ could be infinite.
In this setting the value function definition reads:
\begin{equation}
    \label{fo:continuous_value}
    V(t,x)=\inf_{\balpha}\EE\Bigl[\int_t^T e^{-\beta(s-t)} f\bigl(s,X_s,\alpha_s\bigr)\Bigl| X_t=x\Bigr].
\end{equation}
In full analogy with the discrete-time case, control processes are often given in closed loop Markovian form. In other words, the random action $\alpha_t$ at time $t$ is given by the value $\pi_t(X_t)$ of a feedback function $\pi_t:S\mapsto S$, and the set $\bpi=(\pi_t)_{t\ge 0}$ is often called a \emph{policy}.
In continuous time, the appropriate version of the DPP is technically much more difficult to derive rigorously from a mathematical standpoint. In any case it leads to a form of the Bellman equation which is now a nonlinear partial differential equation (PDE) known as the Hamilton-Jacobi-Bellman (HJB) equation.

Existence of the value function is usually understood as the infimum appearing in formula \eqref{fo:continuous_value} is often finite for every couple $(t,x)$. This is certainly not a very restrictive assumption. Indeed, it is obviously satisfied when $f$ is bounded or merely bounded from below. However, the existence of an admissible control process $\balpha$ realizing the minimum in \eqref{fo:continuous_value} is a more challenging problem. This is very unfortunate because in many applications, the knowledge of an optimal control policy is more important than the knowledge of the value function. Existence of such an optimal control process does not hold in general. The Fillipov convexity condition is the most commonly used sufficient condition guaranteeing existence of an optimal control process. Roughly speaking this condition says that for each couple $(t,x)$, the set $\{(b(t,x,a),f(t,x,a));\,a\in A\}$ is convex.
See for example \cite{ElKaroui_et_al} for a proof.

\section{\textbf{Compactifications in the Search for Optimal Controls}}
\label{se:compactifications}

While finding sufficient conditions for existence of the value functions -- typically proving that the infima are finite -- is relatively straightforward, proving that the infima are attained is usually much more difficult.
Knowing that the value function exists and is finite is interesting in itself. However, in most applications, knowing the  of an optimal control process, and understanding its properties is much more important.
Life would be easier if the search for optimal control processes could be limited to compact spaces, and the restrictions of the objective functions (i.e. expected aggregate discounted future costs) to these spaces were to be lower semi-continuous. Unfortunately, this is rarely the case, and more sophisticated tools need to be brought to bear in order to tackle this challenging issue.

\vskip 2pt
The following two subsections present the time-honored approaches to the compactification processes used in the discrete and continuous time cases to transform the original models in such a way that the \emph{wishful thinking} articulated above becomes reality. We may trace their origins far back in the history of the subject.
The most famous example in the discrete-time case is most likely the existence of Nash equilibria in mixed strategies for finite games, \cite{Nash_1951}, while the introduction of Young measures \cite{Young} is among the most cited early instances in the continuous-time case.

\vskip 6pt
The original purpose of this paper was to examine the differences and the similarities between these parallel compactification approaches. Using mixed controls in the discrete time case, and using relaxed controls in continuous time control, we recast these two approaches in a common mold, and hopefully, connect them in a natural way. This procedure should be reminiscent of the standard approximation procedures identifying the equations of continuous models as limits of their analogues after grid discretizations of the time axis.

\vskip 4pt
The remaining subsections are devoted to the discussion of those recastings in the special cases of the optimization of MDPs introduced earlier in Subsection \ref{sub:MDPs}, and of the optimal control of SDEs in Subsection \ref{sub:SDEs}.

\subsection{Controlling MDPs with Mixed Controls}
\label{sub:mixed_mdps}
Consider an MDP as defined earlier, say with transition probability kernel $P $ instead of a system function $F$ for the sake of the present discussion. In such a set-up, a \emph{mixed control}, or a \emph{randomized action}, is a probability distribution over the action space $A$, namely an element $\alpha=\alpha(d a)$ of $\cP(A)$ the space of probability measures on the action space $A$, instead of a plain element $a\in A$. Clearly, any specific action $a\in A$ can be viewed as a randomized action $\alpha=\delta_a$, the Dirac measure concentrated at the point $a\in A$.

\vskip 2pt
Naturally, we call a \emph{mixed control process} (sometimes called a randomized control process) $\balpha=(\alpha_n)_{0\le n\le T}$, any non-anticipative sequence of random variables with values in $\cP(A)$.  As usual, the non-anticipative property simply means that the choice of $\alpha_n$ can only depend upon the information available at time $n$. Given a mixed control process $\balpha=(\alpha_n)_{0\le n\le T}$, the dynamics of the state process $\bX=(X_n)_{0\le n\le T}$ are given exactly as before by a control process $\ba=(a_n)_{0\le n\le T}$ with the following property: at each time $n$, given the value of the state $X_n$ (and possibly its past values as well as the past values of the actions taken), the value of the random action $a_n$ is independently sampled from the distribution $\alpha_n=\alpha_n(da)$. In other words, $X_{n+1}\sim P(X_n,a_n)$ with $a_n\sim \alpha_n$ independently of the past.

\vskip 4pt
A convenient way to construct such a control process $\ba=(a_n)_{0\le n\le T}$ forming conditionally independent samples from the randomized control process $\balpha=(\alpha_n)_{0\le n\le T}$ is to use the Blackwell-Dubins theorem \cite{BlackwellDubins}. The latter states that if $A$ is a Polish space, there exists a Borel measurable function $\rho_A:\cP(A)\times [0,1]\to A$ such that
for any $\mu\in\cP(A)$, the image (i.e. the push-forward) of the Lebesgue measure on the unit interval $[0,1]$ by the function $[0,1]\ni u\mapsto \rho_A(\mu,u)\in A$ is exactly $\mu$. Moreover, for almost every $u\in[0,1]$ the function $\cP(A)\ni\mu\mapsto \rho_A(\mu,u)\in A$ is continuous for the topology of the weak convergence on $\cP(A)$. We shall use such a function $\rho_A$ and call it the Blackwell-Dubins function of the space $A$ for definiteness. Notice that in fact, the result remains true if one replaces the unit interval $[0,1]$ and its Lebesgue's measure by any probability space. 

\vskip 2pt
The Blackwell-Dubins theorem happens to be extremely convenient in our situation because one can easily construct inductively a sequence $(U_n)_{n\ge 0}$ of uniformly distributed random variables in $[0,1]$ in such a way that for each $n\ge 0$, $U_{n+1}$ is independent of all the $U_m$ for $m\le n$ as well as of the past values of the state, namely the $X_m$ for $m\le n$. Indeed, with such a sequence in hand, one can choose:
\begin{equation}
    \label{fo:alpha_n}
    a_n=\rho_A\bigl(\alpha_n,U_{n+1}\bigr).
\end{equation}
This strategy was repeatedly used in \cite{CarmonaLauriere_AAP} in mean field reinforcement learning with common noise.

\vskip 2pt
In the case of (closed loop) Markovian models, the mixed action $\alpha_n$ is of the form $\alpha_n=\pi_n(X_n)$ for some deterministic function $\pi_n:S\ni x\mapsto\pi_n(x)\in\cP(A)$. Even though it is extensively used, the terminology \emph{Markovian} may be some kind of a stretch in this situation, but we shall conform to this usage for the sake of clarity. As earlier, we call the set $\bpi=(\pi_n)_{n\ge 0}$ a \emph{mixed policy}. They are extensively used in RL applications.

\vskip 4pt
We now argue that the randomization provided by the use of mixed controls is in fact a form of "compactification" as hinted at in the title of the section. Indeed, if for example, we restrict ourselves to bounded sets $A$, then $A$ is compact, and so is $\cP(A)$, and continuity of functions with values in $\cP(A)$ equipped with the topology of the weak convergence of measures is a relatively weak requirement.

\subsection{Continuous Time Control with Relaxed Controls}
\label{sub:standard_relaxed}
Typically, a relaxed control is a continuous time adapted process $\balpha=(\alpha_t)_{0\le t\le T}$ with values in $\cP(A)$. In order to emphasize the compactification effect of such a randomization, it is convenient, at least for finite horizon $T$, to view the (random) value of a relaxed control $\balpha$ as a probability measure on the product space $[0,T]\times\RR^d$ whose first marginal (projection) is a multiple of the Lebesgue measure. Indeed, by disintegration of such a measure one recovers $\balpha=(\alpha_t)_{0\le t\le T}$ as $\balpha(dt,d\alpha)=\alpha_t(d a)dt/T$.
Viewing relaxed controls as random elements in the new action space $\tilde A$ of elements of $\cP([0,T]\times A)$ whose first marginal is a multiple of the Lebesgue measure has interesting benefits. Indeed, this space of measures is compact for the topology of the \emph{stable convergence}, and simple criteria exist for the stable continuity of functions defined on this space. See for example \cite{JacodMemin_stable} or \cite{ElKaroui_et_al} for details.

So when using relaxed controls, for almost every time $t\in [0,T]$, a probability measure $\alpha_t(d a)$ plays the role of the choice of an action. This is very much in the spirit of the mixed controls in the optimization of discrete-time
Markov decision processes. However, the fact that time is now continuous does create difficulties and understanding them is one of the motivations of this paper.

\vskip 2pt
In the abundant literature on relaxed control models, \emph{compactification} of the 
dynamics of the controlled state \eqref{fo:state_sde} gives rise to two separate streams of procedures depending on whether or not the control $\alpha_t$ appears in the volatility coefficient $\sigma$. The case of controlled volatility is mathematically more involved, and most relevant to the thrust of this paper, it is not 
amenable to the RL randomization procedures advocated here. We devote Section \ref{se:discussion} to this set of models, but for the time being, we limit ourselves to models for which the volatility is independent of the control, namely for which $\sigma(t,X_t,\alpha_t)=\sigma(t,X_t)$.

For these models, the relaxed control randomization of the dynamics is given by an SDE of the form
\begin{equation}
    \label{fo:relaxed_dynamics}
    d\tilde X_t = \tilde b(t,\tilde X_t, \alpha_t) dt + \tilde\sigma(t,\tilde X_t) dW_t
\end{equation}
where the new coefficients are given by
\begin{equation}
    \label{fo:relaxed_coefs}
    \tilde b(t,x, \alpha)  = \int_A b(t, x,a) \alpha(da) ,
    \quad\text{and}\quad 
    \tilde\sigma(t,x, \alpha) = \sigma(t,x),
\end{equation}
in terms of the coefficients of the original formulation of the control problem.
If we define the new instantaneous cost function $\tilde f$ by
\begin{equation}
    \label{fo:relaxed_cost_function}
    \tilde f(t,x, \alpha)  = \int_A f(t, x,a) \alpha(da) ,
\end{equation}
the new relaxed expected cost reads:
\begin{equation}
    \label{fo:relaxed_cost}
J(\bpi)=    \EE\Bigl[\int_0^T e^{-\beta t} \tilde f\bigl(t, \tilde X_t, \alpha_t\bigr) \;dt\Bigr]=    \EE\Bigl[\int_0^T e^{-\beta t} \int f\bigl(t, \tilde X_t, a \bigr)\alpha_t(da) \;dt\Bigr],
\end{equation}
where as mentioned earlier, the time horizon $T$ could be infinity, i.e. $T=\infty$, and the discount factor could be zero, i.e. $\gamma=1$ or $\beta=0$. So we now face a new stochastic control problem given by the data $(\tilde b, \tilde \sigma,\tilde f)$ for which we denote the value function by $\tilde V$. Interestingly enough, the new coefficients are linear in the new control variable $\alpha$, and given the compactness of the new space of actions $\tilde A$, existence of an optimal relaxed control realizing the infimum defining the value function is not a surprise. See \cite{ElKaroui_et_al} for example. Still, the new coefficients $\tilde b$ and $\tilde \sigma$ given by formulas \eqref{fo:relaxed_coefs} being possibly different from the original $b$ and $\sigma$, the new dynamics are likely to be different from the original dynamics, and one may wonder how the new value function relates to the original one. Now, restricting ourselves to relaxed controls of the form $\alpha_t(da)=\delta_{a_t}(da)$ for admissible control processes $\ba=(a_t)_{0\le t\le T}$, we see that not only do the dynamics coincide in this case, but also the values do coincide as $J(\ba)=J(\balpha)$. Consequently, this implies that $\tilde V(t,x)\le V(t,x)$. What is remarkable is that equality actually holds. This is the major merit of the compactification via relaxed controls. Equality of the values functions is a consequence of a result known as the \emph{chattering lemma} which states that for each relaxed control $\balpha=(\alpha_t)_{0\le t\le T}$, there exists a sequence $(\ba^n)_{n\ge 1}$ of (pure) admissible controls such that the relaxed controls $\balpha^n$ defined by $\alpha^n_t(da)=\delta_{a^n_t}(da)$ approximate $\balpha$ in the sense of the stable topology, and 
$$
\tilde J(\balpha)=\lim_{n\to\infty}\tilde J(\balpha^n)
=\lim_{n\to\infty} J(\ba^n).
$$

\vskip 12pt
Since the practice in reinforcement learning is to use policies in closed loop form, for the sake of presentation, we shall restrict ourselves to relaxed controls in Markovian feedback forms, for which $\alpha_t=\pi_t(X_t)$ for a deterministic function $\pi_t: S\to \cP(A)$, the set $\bpi=(\pi_t)_{0\le t\le T}$ of deterministic functions being called a policy. The chattering lemma being tailor-made for open-loop models, this choice will require ad hoc approximation arguments to connect optimization with relaxed control with the mixed control practice of reinforcement learning.

\vskip 2pt
Notice that as a general rule, we try to use the notation $\bpi=(\pi_t)_t$ for policies (which can also be viewed as feedback functions), $\balpha=(\alpha_t)_t$ for control processes, and $\ba=(a_t)_t$ for the actual action processes.

\section{\textbf{First Attempt at the Randomization of a Continuous Time Model}}
\label{se:after_sun}
As already mentioned in the introduction, a heuristic argument was given in \cite[Section 2.1]{Wang_Z_Z_RL} to try to connect the discrete and continuous time randomizations/relaxations. However, their argument is informal and relies on the use of a large number of independent copies of the model, allowing the authors to appeal to the standard law of large numbers. Our approach is different because, not only do we want to rely on rigorous mathematical arguments and complete proofs, but we also do not want to rely on the unnatural and unrealistic assumption of the availability of a large number of independent copies of the model. Also, in a follow-up paper \cite{JiaZhou}, some of those authors try to appeal to the so-called \emph{exact Law of Large Numbers} for a different justification, but again, this argument falls short of a rigorous rationale, and we debunk it below in Subsection \ref{sub:warming-up}.

\subsection{What Could a Randomization with Mixed Controls Look Like?}
\label{sub:goal}
Here we discuss the natural way to implement the randomization of MDPs with mixed controls presented in Subsection \ref{sub:mixed_mdps} in the case of continuous-time models in hope to recover the formulation of Subsection \ref{sub:standard_relaxed} using relaxed controls.
The generalization of the mixed policy approach to the control of MDPs would be to consider  a continuous time policy $\bpi=(\pi_t)_{0\le t\le T}$ where for each $t\in[0,T]$, $\pi_t:S\to\cP(A)$ would associate to each state of the system $x\in S$, a probability distribution $\pi_t(x)$ over the set of possible actions at time $t$. In full analogy with the discrete time case, such a continuous time policy would lead to a \emph{randomized} state evolution consistent with the dynamics \eqref{fo:state_sde} if one could extract a control process $\ba=(a_t)_{0\le t\le T}$ for which, at each time $t$, $a_t$ would be a random element of $A$ with probability distribution $\alpha_t=\pi_t(X_t)$, independent of the past history of the system, and in particular, independent  of all the $a_s$ for $0\le s<t$.
As a result, $\ba=(a_t)_{0\le t\le T}$ would be a continuous-time white noise, raising highly technical measurability issues whose resolutions, when possible, are rather sophisticated and dramatically reduce the domain of application of the arguments.

\subsection{Warm-Up with a Simple Example}
\label{sub:warming-up}

Here we consider a simple model in which the volatility of the state dynamics is a deterministic constant $\sigma>0$, the set $A$ of possible actions is a closed convex subset of $\RR^d$, and the drift vector field is exactly equal to the control, i.e. $b(t,x,a) = a$. Furthermore, we  limit ourselves to policies and feedback functions which only depend upon time. In this case, the drift vector field of the relaxed control dynamics becomes:
$$
\tilde b(t,x,\alpha)=\int_A a\;\alpha(da)
$$
which is the mean (average) $\bar\alpha$ of the probability measure $\alpha\in\cP(A)$ used as a new control in the new action space $\cP(A)$. In this case, for each relaxed control $\bpi=(\pi_t)_{0\le t\le T}$, the state in the model with relaxed controls reads
\begin{equation}
    \label{fo:tilde_X_t}
\tilde X_t=X_0+\int_0^t\bar\pi_s\; ds+\sigma W_t,
\end{equation}
and according to the above informal discussion, the mixed control approach to such a model should involve a time evolution of the form
\begin{equation}
    \label{fo:X_t}
X_t=X_0+\int_0^t\alpha_s\; ds+\sigma W_t,
\end{equation}
where $\balpha=(\alpha_t)_{0\le t\le T}$ is a measurable process such that for each $t\in[0,T]$, the law of each $\alpha_t$ is $\pi_t$ and the $A$-valued random variable $\alpha_t$ is independent of $(\pi_s)_{0\le s< t}$ and $(\alpha_s)_{0\le s<t}$. 

\vskip 6pt
It is well known that the measurability of $t\mapsto \alpha_t$ and the white noise property of independence of the $\alpha_t$'s among themselves are two incompatible requirements. The closest we can get to satisfying them simultaneously is to work with so-called Fubini extensions, see e.g. \cite{Sun,CarmonaDelarue_book_I,carmona2022stochastic,aurell2022stochastic,aurell2022finite}.
Still, because those concepts were informally mentioned in the discussion of \cite{JiaZhou}, and despite the fact that our conclusions will fall short of the goals we set up for ourselves, 
we review the mathematical definitions and the main steps of a possible resolution of this measurability dilemma. However, as mentioned at the top of this subsection, 
we need to assume that the relaxed control processes $\bpi=(\pi_t)_{0\le t\le T}$ are deterministic for us to be able to use the results of the Fubini's extensions theory.
Here, for the sake of completeness, we recall the definitions and the results we shall use in the remainder of this subsection.
\begin{definition}
\label{de:essentially_iid}
The $A$-valued random variables $(\alpha_s)_{0\le s\le T}$ are said to be essentially pairwise independent if, for almost every $s\in[0,T]$, the random variable $\alpha_s$ is independent of $\alpha_t$ for almost every $t\in[0,T]$. 
\end{definition}

One may wonder if  weakening the full independence requirement to essentially pairwise independence is enough for the construction of such processes on a given probability space $(\Omega,\cF,\PP)$ so that the process $\balpha:[0,T]\times\Omega\ni(t,\omega)\mapsto \alpha_t(\omega)$ satisfies relevant measurability properties. The answer is yes if one is willing to construct such processes on specific extensions of the product space $([0,T]\times \Omega,\cB_{[0,T]}\otimes\cF,\lambda\otimes\PP)$ where $\lambda$ stands for the Lebesgue measure. Such extensions are called Fubini's extensions. In fact, some of these extensions are \emph{rich} enough to allow for the construction of essentially pairwise independent processes with arbitrary chosen distributions, leading to the following statement borrowed from \cite{Sun}.
 
\begin{proposition}
\label{pr:rich_white_noise}
For $\;\Omega=C([0,T];\RR^d)$ equipped with its Borel $\sigma$-field $\cF$ and Wiener measure $\PP$,  there exists a Fubini extension $([0,T]\times\Omega,\cB_{[0,T]}\boxtimes\cF,\lambda\boxtimes\PP)$ with the Lebesgue measure $\lambda$ such that,
$A$ being a Polish space, for any deterministic measurable function $\bpi:[0,T]\to\cP(A)$,  there exists a $\cI\boxtimes\cF$-measurable $A$-valued essentially pairwise independent process $\balpha:[0,T]\times\Omega\to A$ such that for almost every $t\in [0,T]$, the law of $\alpha_t$ is $\pi_t$.
\end{proposition}
So for each deterministic relaxed control $\bpi=(\pi_t)_{0\le t\le T}$ the state dynamics controlled by such a relaxed control, as given in \eqref{fo:tilde_X_t}, can be 
implemented via a procedure inspired by the independent sampling of mixed controls in the discrete time case, using an $A$-valued measurable controlled process $\balpha=(\alpha_t)_{0\le t\le T}$ controlling state dynamics \eqref{fo:X_t}. Finally, the fact that the stochastic dynamics \eqref{fo:tilde_X_t} can be identified to \eqref{fo:X_t} is a consequence of the \emph{exact law of large numbers} which, in the present set up, states that 
$$
    \int_0^t\alpha_s ds= \int_0^t\EE[\alpha_s]\; ds=\int_0^t\bar\pi_s\;ds.
$$

\begin{remark}[Important Remark]
\label{re:fubini_extension}
    Before closing this subsection, we want to emphasize that the results it contains should only be viewed as an \emph{expedient} given the stated objectives of our analysis. Indeed, even for the special case of deterministic relaxed controls, they fall short of what we were hoping for as outlined in Subsection \ref{sub:goal}. First, the measurability of the process $\balpha=(\alpha_t)_{0\le t\le T}$ only holds for the $\sigma$-field $\cB_{[0,T]}\boxtimes\cF$ which is an (obscure) enlargement of the product $\sigma$-field $\cB_{[0,T]}\otimes\cF$. Second, $\alpha_t$ is guaranteed to have the desired distribution $\pi_t$, only for almost every $t\in[0,T]$. Third, the $\alpha_t$'s do not form a white noise since they are not really independent, but only \emph{essentially pairwise independent}, and this is a serious issue. And fourth, we do not know if the random variables $\alpha_t$ are independent of the past history $X_{[0,t]}$ of the state. 
    
    Finally, the existence of the process $\balpha=(\alpha_t)_{0\le t\le T}$ is merely the result of an abstract mathematical construct, and as far as we know, there is  no constructive proof of this existence.
\end{remark}

\section{\textbf{Continuous Time Relaxed Models as Limits of Mixed MDP Models}}
\label{se:final}

This section contains the main technical result of the paper. We show that the state process controlled by a feedback relaxed policy in continuous time is the limit of a sequence of discrete-time state processes controlled at the times of a finite grid, by mixed controls sampled from the time discretization of the original relaxed control policy. It is important to emphasize that our approximation result is in the category of \emph{strong approximation} results in the sense that it is at the level of the trajectories. This is of the utmost importance for practical implementations, and it is one of the reasons why we restrict ourselves to It\^o stochastic differential equations in which the volatility is not controlled.

Not surprisingly, we follow the time-honored method of approximating continuous-time stochastic models by discretizations of the time axis. After all, when one is numerically solving a continuous-time stochastic control problem, whether the implementation uses the knowledge of the model, or is \emph{model-free} and based on Reinforcement Learning algorithms, the first step is always to discretize time. So starting from a relaxed control policy in Markovian feedback form, we use the Euler-Maruyama scheme to reset the model on a finite time grid.
In order to guarantee the convergence of the Euler-Maruyama approximation, we assume that the drift vector field $b$ and the volatility matrix $\sigma$ are $1/2$-H\"{o}lder continuous in time and Lipschitz continuous in the state and action variables. Also, we  restrict ourselves to relaxed controls with the same regularity properties. In fact, the statement and the proof below are given for time independent coefficients for the sake of simplicity.

The convergence of the state trajectories being uniform in time, we derive as a by-product, the convergence of the corresponding expected aggregated costs.

\vskip 6pt
In the following proof, we use the following simple discrete version of Gronwall's inequality
\begin{equation}
    \label{fo:Gronwall}
u_j\le \alpha +\sum_{h=0}^{j-1}\beta_h u_h\qquad\Longrightarrow\qquad
u_j\le \alpha\exp\Bigl(\sum_{h=0}^{j-1}\beta_h\Bigr).
\end{equation}
For the sake of definiteness, we consider only a model in finite time horizon $T$. 

\begin{theorem}
\label{thm:main-X}
Assume that $b$ and the policy $\pi$ do not depend on time, and $\sigma$ does not depend on the action. Assume that $b$ is bounded, and that $b$ and $\pi$ are Lipschitz continuous. Consider the \emph{relaxed state}:
\begin{equation}
    \label{fo:continuous_dynamics}
          X_t = X_0 +\int_0^t\tilde b\bigl( X_s,\pi(X_s)\bigr)ds+ \int_0^t\tilde \sigma\bigl( X_s\bigr) dW_s
\end{equation}
where the relaxed coefficients are given by:
\begin{equation}
    \label{fo:relaxed_coefs}
\tilde b(x,\pi)=\int_A b(x,a)\pi(x)(da)\qquad\text{and}\qquad \tilde \sigma(x)=\sigma(x).   
\end{equation}

For each positive integer $N$, consider the time discretization grid $\{t_n\}_{n=0,1,\cdots,N}$
with $t_n=nT/N$. 
We consider the following dynamics for the process with \emph{randomized actions}: 
\begin{equation}
\label{fo:X_hat_MN}
    \hat X_{t_n}
    =X_0+\frac{T}{N}\sum_{m=0}^{n-1} b\bigl(\hat X_{t_m},\hat a_m \bigr)
    +\sqrt{\frac{T}{N}}\sum_{m=0}^{n-1}\sigma\bigl(\hat X_{t_m}\bigr)\;\epsilon_{m+1},
\end{equation}
where $\epsilon_{m+1}=\sqrt{N/T} [W_{t_{m+1}}-W_{t_m}]$ and each random variables $\hat a_m$ is independent of the past sources of randomness up to (and including) time $t_m$ and $\hat a_m\sim \pi\bigl(\hat X_{t_m}\bigr)$, i.e., is distributed according to the probability distribution $\pi\bigl(\hat X_{t_m}\bigr)$. Then,
    \[
    \EE\left[\max_{0\le n\le N}| X_{t_n} - \hat X_{t_n} |^2 \right] ^{1/2}
    \le\frac{C}{\sqrt{N}} ,
\]
where the constant $C$ depends on $T$ and the data of the problem, including the policy $\pi$. %
\end{theorem}
As mentioned earlier, the dynamics \eqref{fo:X_hat_MN} are motivated by the mixed control implementation of RL algorithms.

\begin{remark}
    The proof given below for time independent coefficients $b$ and $\sigma$ and time independent policies $\pi$ works as well in the time dependent case as long as the Lipschitz assumptions are uniform in time and the time dependence is H\"older of order $1/2$.
Also the result on the control of the quadratic norm can be extended to other powers $p\ne 2$ by using the Burkholder-Davis-Gundy inequality instead of Doob's inequality.
\end{remark}

\begin{proof}[Proof of Theorem~\ref{thm:main-X}]

\vskip 6pt\noindent\textbf{Step 1: Preliminaries. }
We fix an integer $N\ge 1$. The standard way to approximate the solution of an SDE like \eqref{fo:continuous_dynamics} by the solution of a finite difference equation over the finite time grid $\{t_n\}_{n=0,1,\cdots,N}$ is to use the Euler-Maruyama approximation.
So we consider the process $\tilde \bX=(\tilde X_t)_{0\le t\le T}$ defined by:
$$
    \tilde X_t=X_0+\int_0^t\tilde b\bigl(\tilde X_{t_{n(s)}},\pi(\tilde X_{t_{n(s)}})\bigr) ds + \int_0^t\tilde \sigma\bigl(\tilde X_{t_{n(s)}}\bigr) dW_s,
$$
where  $t_{n(s)}=t_n$ if $t_n\le s < t_{n+1}$. The approximating process is a continuous time process, given by a form of interpolation of
the sequence $\bY=(Y_n)_{0\le n\le N}$ of random states defined recursively by $Y_0=X_0$ and the equation
$$
    Y_{n+1}=Y_n+\frac{T}{N} \tilde b\bigl(Y_n,\pi(Y_n)\bigr) + \tilde \sigma\bigl(Y_n\bigr)[W_{t_{n+1}}-W_{t_n}],\qquad n=0,1,\cdots,N-1,
$$
giving proxies for the values of the approximating process at the points of the subdivision given by the times $t_n$.
Indeed, for $0\le n\le N-1$ we have
\begin{equation}
\label{eq:relaxed-discrete-time}
    \begin{split}
    \tilde X_{t_{n+1}}&=\tilde X_{t_n}+\int_{t_n}^{t_{n+1}}\tilde b\bigl(\tilde X_{t_n},\pi(\tilde X_{t_n})\bigr) ds + \int_{t_n}^{t_{n+1}}\tilde \sigma\bigl(\tilde X_{t_n}\bigr) dW_s
    \\
    &=\tilde X_{t_n}+\frac{T}{N}\tilde b\bigl(\tilde X_{t_n},\pi(\tilde X_{t_n})\bigr)  + \tilde \sigma\bigl(\tilde X_{t_n}\bigr) [W_{t_{n+1}}-W_{t_n}],
\end{split}
\end{equation}
which coincides with $Y_{n+1}$ as defined above.

It is well known that if the drift vector field $x\mapsto \tilde b\bigl(x,\pi(x)\bigr)$ and the volatility matrix $x \mapsto \tilde \sigma(x)$
are Lipschitz continuous, then the Euler-Maruyama scheme has a strong rate of convergence of order $1/2$, in the sense that for any $p > 0$, there exists a constant $C_p > 0$ such that
\begin{equation}
\label{eq:rate-relaxed-time-discretization}
\EE\bigl[ \sup_{0\le t\le T}|X_t-\tilde X_t|^p\bigr]\le \frac{C_p}{N^{-p/2}}.
\end{equation}

As defined above, using the notation $\epsilon_{n+1}=\sqrt{N/T} [W_{t_{n+1}}-W_{t_n}]$,
\begin{equation}
\label{fo:X_tilde_n}
    \tilde X_{t_n}
    =\tilde X_0+\frac{T}{N}\sum_{m=0}^{n-1}\int_A b\bigl(\tilde X_{t_m},a\bigr)\pi(\tilde X_{t_m})(da)
    +\sqrt{\frac{T}{N}}\sum_{m=0}^{n-1}\sigma\bigl(\tilde X_{t_m}\bigr)\;\epsilon_{m+1}.
\end{equation}
However, as emphasized throughout the paper, as well as in the statement of the theorem, this is not the practical RL implementation of the state dynamics.
The dynamics relevant to the mixed control based approach to RL are:
\begin{equation}
\label{fo:X_hat_n}
    \hat X_{t_n}
    =\hat X_0+\frac{T}{N}\sum_{m=0}^{n-1} b\bigl(\hat X_{t_m},\hat a_m \bigr)
    +\sqrt{\frac{T}{N}}\sum_{m=0}^{n-1}\sigma\bigl(\tilde X_{t_m}\bigr)\;\epsilon_{m+1},
\end{equation}
where each random variable $\hat a_m$ is independent of the past up to (and including) time $t_m$, and $\hat a_m\sim \pi\bigl(\hat X_{t_m}\bigr)$, i.e. $\hat a_m$ is distributed according to the probability distribution $\pi\bigl(\hat X_{t_m}\bigr)$.

\vskip 12pt\noindent
\textbf{Step 2: Comparison of $\hat \bX$ and $\tilde \bX$. }

From the definitions \eqref{fo:X_hat_n} and \eqref{fo:X_tilde_n} of the processes we get
\begin{equation}
\label{fo:dif0}
\begin{split}
    \bigl|\tilde X_{t_n}- \hat X_{t_n}\bigr|^2
    &\le \frac{4T^2}{N^2} \Bigl|\sum_{m=0}^{n-1}\int_A [b\bigl(\tilde X_{t_m}, a\bigr)-b\bigl(\hat X_{t_m}, a\bigr)]\pi\bigl(\tilde X_{t_m}\bigr)(da)\Bigr|^2\\
    &\hskip 45pt
    +\frac{4T^2}{N^2} \Bigl|\sum_{m=0}^{n-1}\int_A b\bigl(\hat X_{t_m}, a\bigr)[\pi\bigl(\tilde X_{t_m}\bigr)-\pi\bigl(\hat X_{t_m}\bigr)](da)\Bigr|^2 \\
    &\hskip 45pt
    +\frac{4T^2}{N^2} \Bigl|\sum_{m=0}^{n-1}\Bigl(\int_A b\bigl(\hat X_{t_m}, a\bigr)\pi\bigl(\hat X_{t_m}\bigr)(da)-b\bigl(\hat X_{t_m},\hat a_m\bigr)\Bigr)\Bigr|^2\\
    &\hskip 45pt
    +\frac{4T}{N} \Bigl|\sum_{m=0}^{n-1}\Bigl(\sigma(\tilde X_{t_m})-\sigma(\hat X_{t_m})\Bigr)\;\epsilon_{m+1}\Bigr|^2\\
    &\le (i) + (ii) + (iii)+(iv).
    \end{split}
\end{equation}
Notice that
\begin{equation}
\label{fo:(i)}
\begin{split}
    (i)
    &\le \frac{4T^2}{N^2}n\sum_{m=0}^{n-1}\Bigl|\int_A [b\bigl(\tilde X_{t_m}, a\bigr)-b\bigl(\hat X_{t_m}, a\bigr)]\pi\bigl(\tilde X_{t_m}\bigr)(da)\Bigr|^2\\
    &\le \frac{4T^2\|b\|^2_{Lip}}{N}\sum_{m=0}^{n-1}\bigl|\tilde X_{t_m}-\hat X_{t_m}\bigr|^2
\end{split}
\end{equation}
and similarly
\begin{equation}
\label{fo:(ii)}
    (ii)\le \frac{4T^2\|b\|^2_\infty \|\pi\|^2_{Lip}}{N}\sum_{m=0}^{n-1}\bigl|\tilde X_{t_m}-\hat X_{t_m}\bigr|^2
\end{equation}
because of the Lipschitz assumption on $b$ and $\pi$, and the boundedness of $b$.
For the sake of definiteness, for each $t\in[0,T]$ we denote by $\cF_t$ the $\sigma$-field generated by all the sources of randomness prior to (and including) time $t$. Now, in order to control $(iii)$, we introduce a martingale $(Z_n)_n$ to which we shall apply Doob's inequality. For each $n\in\{0,1,\cdots,N-1\}$ we define
$$
    d_n=b\bigl(\hat X_{t_n},\hat a_n\bigr) - \int_A b\bigl(\hat X_{t_n}, a\bigr)\pi\bigl(\hat X_{t_n}\bigr)(da),
    \qquad\text{and}\qquad
    Z_n=\sum_{m=0}^{n-1} d_m, \quad n=1,\cdots,N.
$$
With this notation:
\begin{equation}
\label{fo:(iii)}
    (iii)=\frac{4T^2}{N^2}\bigl|Z_n\bigr|^2.
\end{equation}
$(Z_n)_{1\le n\le N}$ is indeed a $(\cF_{t_n})_{1\le n\le N}$ martingale since:
$$
    \EE[Z_{n+1}|\cF_{t_n}]=Z_n + \EE[d_n|\cF_{t_n}]=Z_n
$$
because $\hat a_n$ is independent of $\cF_{t_n}$ and conditioned on $\cF_{t_n}$, it is distributed according to $\pi(\hat X_{t_n})$. Using Doob's inequality we get 
\begin{equation}
    \label{fo:Doob}
    \EE\Bigl[\sup_{1\le n\le N}\bigl|Z_n\bigr|^2]\le 4 \EE[|Z_N|^2]=\EE[\sum_{n=1}^N|d_n|^2]\le 16 N \|b\|_\infty^2.
\end{equation}
Now, if we set 
$$
    V_n=\sum_{m=0}^{n-1}\Bigl(\sigma(\tilde X_{t_m})-\sigma(\hat X_{t_m})\Bigr)\;\epsilon_{m+1},\qquad n\ge 1,
$$
then 
$\bV=(V_n)_n$ is a martingale because $\epsilon_{m+1}$ is an increment of the Brownian motion into the future at time $t_m$. Now
\begin{equation}
    \label{fo:(iv)}
    (iv)= \frac{4T}{N}|V_n|^2\le \frac{4T}{N}\sup_{1\le m\le n}|V_m|^2
\end{equation}
and denoting by $[V]_n$ the quadratic variation of the discrete time martingale $\bV=(V_n)_n$, Doob's inequality gives:
\begin{equation}
    \label{fo:Doob2}
\begin{split}
    \EE\bigl[\sup_{1\le m\le n}|V_m|^2\bigr]
    &\le 4\EE\bigl[[V]_n^2\bigr]\\
    &\le 4 \EE\Bigl[\sum_{m=0}^{n-1}\text{trace}\bigl[\bigl(\sigma(\tilde X_{t_m})-\sigma(\hat X_{t_m})\bigr)^t\bigl(\sigma(\tilde X_{t_m})-\sigma(\hat X_{t_m})\bigr)\bigr]\Bigr]\\
    &\le 4\|\sigma\|_{Lip}^2\sum_{m=0}^{n-1}\EE[|\tilde X_{t_m}-\hat X_{t_m}|^2].
\end{split}
\end{equation}
From now on, in order to simplify the notation, we use the notation $c$ for a positive constant which can change from line to line, but which only depends upon $T$, $\|b\|_\infty$, $\|b\|_{Lip}$ and $\sigma\|_{Lip}$, and most importantly, which is independent of $N$.
Coming back to \eqref{fo:dif0} and using \eqref{fo:(i)}, \eqref{fo:(ii)}, \eqref{fo:(iii)} and \eqref{fo:(iv)} we get:
\begin{equation}
\label{fo:almost_there}
    \begin{split}        
\sup_{0\le \ell\le n}\bigl|\tilde X_{t_\ell}- \hat X_{t_\ell}\bigr|^2
&\le \frac{c}{N} \sum_{m=0}^{n-1}\sup_{0\le \ell\le m}\bigl|\tilde X_{t_\ell}-\hat X_{t_\ell}\bigr|^2+\frac{c}{N^2}\sup_{1\le n\le N}|Z_n|^2 + \frac{c}{N}\sup_{1\le m\le n}|V_m|^2
    \end{split}
\end{equation}
Taking expectations of both sides and using \eqref{fo:Doob} and \eqref{fo:Doob2} to control the two right most terms, we get:
\begin{equation}
\label{fo:pre_Gronwall}
    \begin{split}        
\EE\Bigl[\sup_{0\le \ell\le n}\bigl|\tilde X_{t_\ell}- \hat X_{t_\ell}\bigr|^2\Bigr]
&\le \frac{c}{N} \sum_{m=0}^{n-1}\EE\Bigl[\sup_{0\le \ell\le m}\bigl|\tilde X_{t_\ell}-\hat X_{t_\ell}\bigr|^2\Bigr]
+\frac{c}{N} + \frac{c}{N}\sum_{m=0}^{n-1}\EE\Bigl[\sup_{0\le \ell\le m}|\tilde X_{t_\ell}- \hat X_{t_\ell}\bigr|^2\Bigr].
    \end{split}
\end{equation}
Using Gronwall's inequality we get:
\begin{equation}
\label{fo:final}       
    \EE\Bigl[\sup_{0\le \ell\le N}\bigl|\tilde X_{t_\ell}- \hat X_{t_\ell}\bigr|^2\Bigr]
    \le \frac{c}{N},
\end{equation}
which is the desired result
\end{proof}

The above strong approximation result has the following consequence for the convergence of the expected costs.

\begin{corollary}
    Under the above assumptions, if the running and terminal cost functions $f$ and $g$ are uniformly continuous, then we have:
$$
\lim_{N\to\infty}\EE\Bigl[g(\hat X_{t_N})+\frac{T}{N}\sum_{n=0}^{N-1}f(t_n,\hat X_{t_n},\hat a_n)\Bigr]
=\EE\Bigl[g( X_{T})+\int_0^T f(t, X_t,\pi_t(X_t))dt\Bigr].
$$
\end{corollary}

\section{\textbf{Numerical experiments}}
\label{se:numeric}

In this section, for the sake of illustration, we present experiments that compare the state process with relaxed control and the state process with mixed randomized actions. Since we cannot simulate processes in continuous time, instead of simulating the process with relaxed dynamics \eqref{fo:continuous_dynamics}, we simulate its discrete-time approximation~\eqref{eq:relaxed-discrete-time}. Recall that the error induced by this approximation is bounded in~\eqref{eq:rate-relaxed-time-discretization}. To simulate this discrete-time process, we compute the integral in the drift (and volatility, if needed) using Gauss-Hermite quadrature with $10$ points.
We consider three different settings:
\begin{itemize}
	\item {\bf Setting 1:} $b(x,a) = a$, $\sigma(x,a) = 0.1$,
	\item {\bf Setting 2:} $b(x,a) = \tanh(a)$, $\sigma(x,a) = 0.1$,
	\item {\bf Setting 3:} $b(x,a) = \tanh(a)$, $\sigma(x,a) = 0.1 + c^a_\sigma a$ where $c^a_\sigma$ is a constant.
\end{itemize}
Note that the third setting does not fit the assumptions of Theorem~\ref{thm:main-X} because in this setting the volatility is controlled. 
For the policy, we consider $\pi(x) = \mathcal{N}(1-x, \sigma_\pi)$ where $\sigma_\pi$ is a constant. The mean is chosen so that in the first setting above, the dynamics attracts the state towards $x=1$. We take $T=5$ and vary the number of steps $N$.

{\bf Setting 1. } Figure~\ref{fig:cont-disc-RL-linear-b-fit} shows that the linear (in $1/N$) upper bound for
$$
    \EE\Bigl[\sup_{0\le \ell\le N}\bigl|\tilde X_{t_\ell}- \hat X_{t_\ell}\bigr|^2\Bigr]
$$
proven in \eqref{fo:final}, is sharp. Using a log-scale, we fit a line and observe that the slope is close to $-1$, which matches the upper bound in~\eqref{fo:final}. The values of the expected maximum square error were obtained by averaging $10$ runs, each expectation being computed with $10000$ trajectories. Figure~\ref{fig:cont-disc-RL-linear-b-fit} also displays the standard deviation as a shaded area around the curve. However, this area is hard to detect because the standard deviation is on the order of $10^{-4}$. 

\begin{figure}[H]
    \centering
    \includegraphics[width=0.48\linewidth]{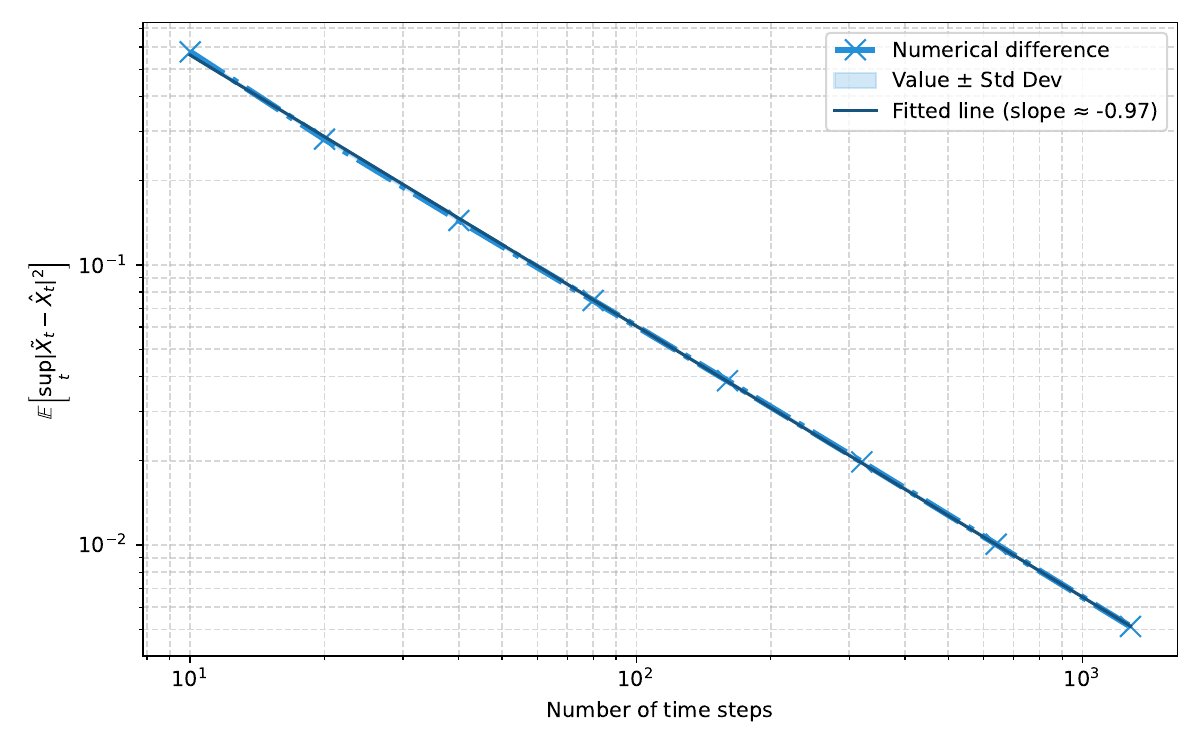}
    \caption{Setting 1. Convergence rate.}
    \label{fig:cont-disc-RL-linear-b-fit}
\end{figure}
Figure~\ref{fig:cont-disc-RL-linear-b-traj} shows four couples of trajectories of the processes $\tilde \bX$ and $\hat\bX$ with $N=100$ in the left pane and $N=1000$ in the right pane, using the same initial points and the same Brownian increments: the Brownian increments for $N=100$ are obtained by summing the increments in the case $N=1000$. As expected, with smaller time steps, the trajectories are closer.

\begin{figure}[H]
    \centering
    \includegraphics[width=0.48\linewidth]{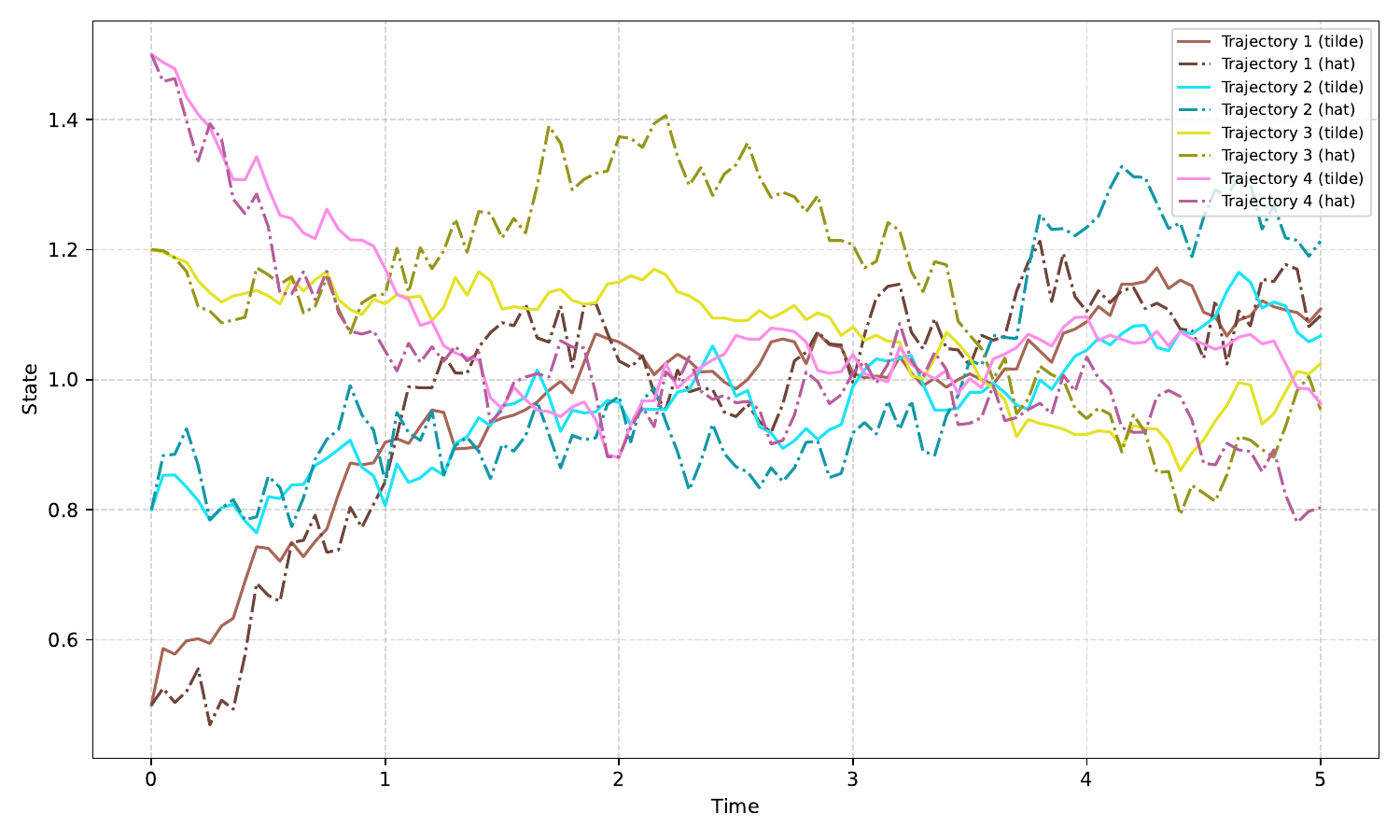}%
    \includegraphics[width=0.48\linewidth]{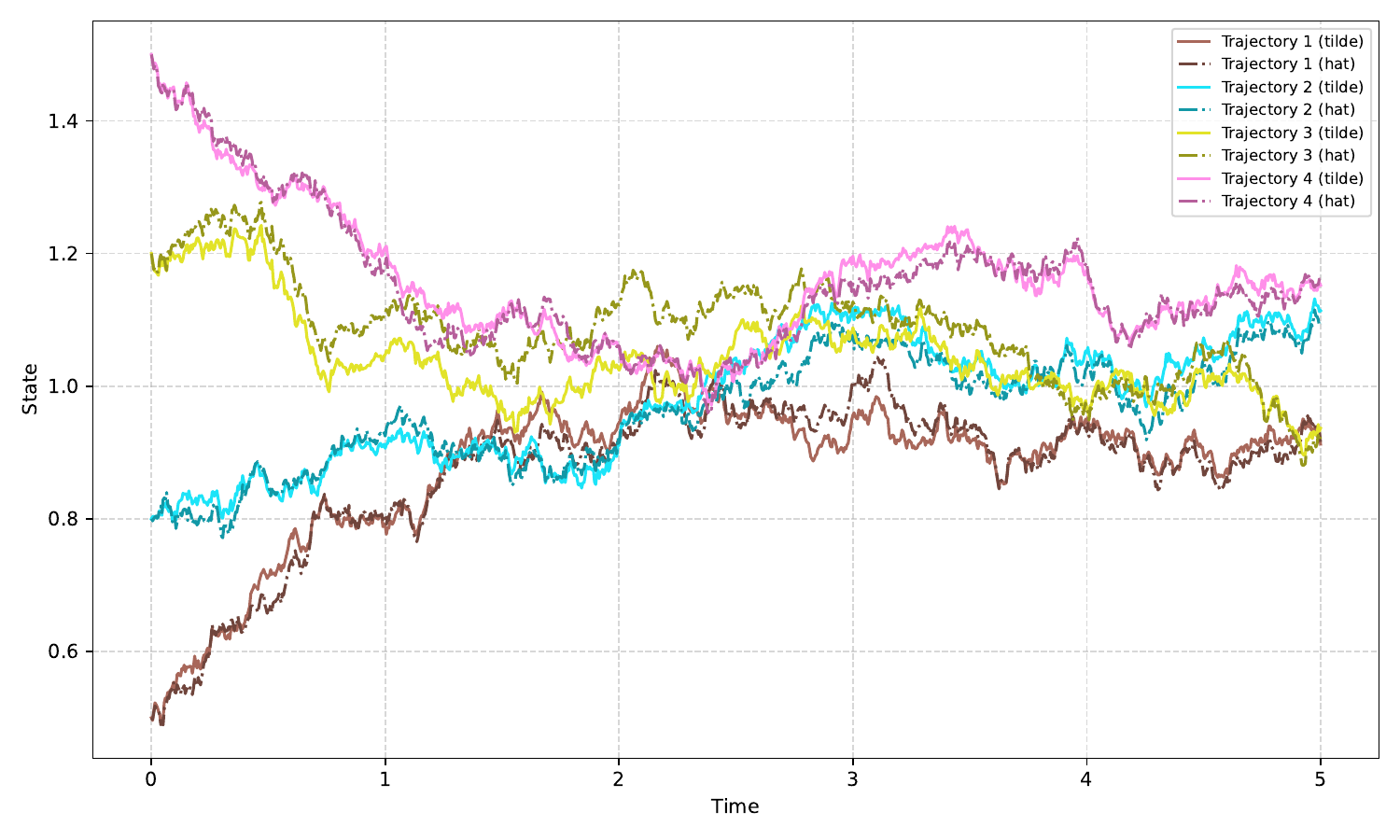}
    \caption{Setting 1. $4$ samples of couples of trajectories with $N=100$ (left) and $N=1000$ (right).}
    \label{fig:cont-disc-RL-linear-b-traj}
\end{figure}

{\bf Setting 2. } In the same way, the left plot in Figure~\ref{fig:cont-disc-RL-tanh-b-fit-traj} confirms the accuracy of the upper bound \eqref{fo:final}. Here again we observe a rate of convergence which is very close to the theoretical upper bound. As before, the behavior of four couples of trajectories with $N=1000$ are displayed in the right plot of Figure~\ref{fig:cont-disc-RL-tanh-b-fit-traj} using the same initial values, and Brownian increments as in Figure~\ref{fig:cont-disc-RL-linear-b-traj}.

\begin{figure}[H]
    \centering
    \includegraphics[width=0.48\linewidth]{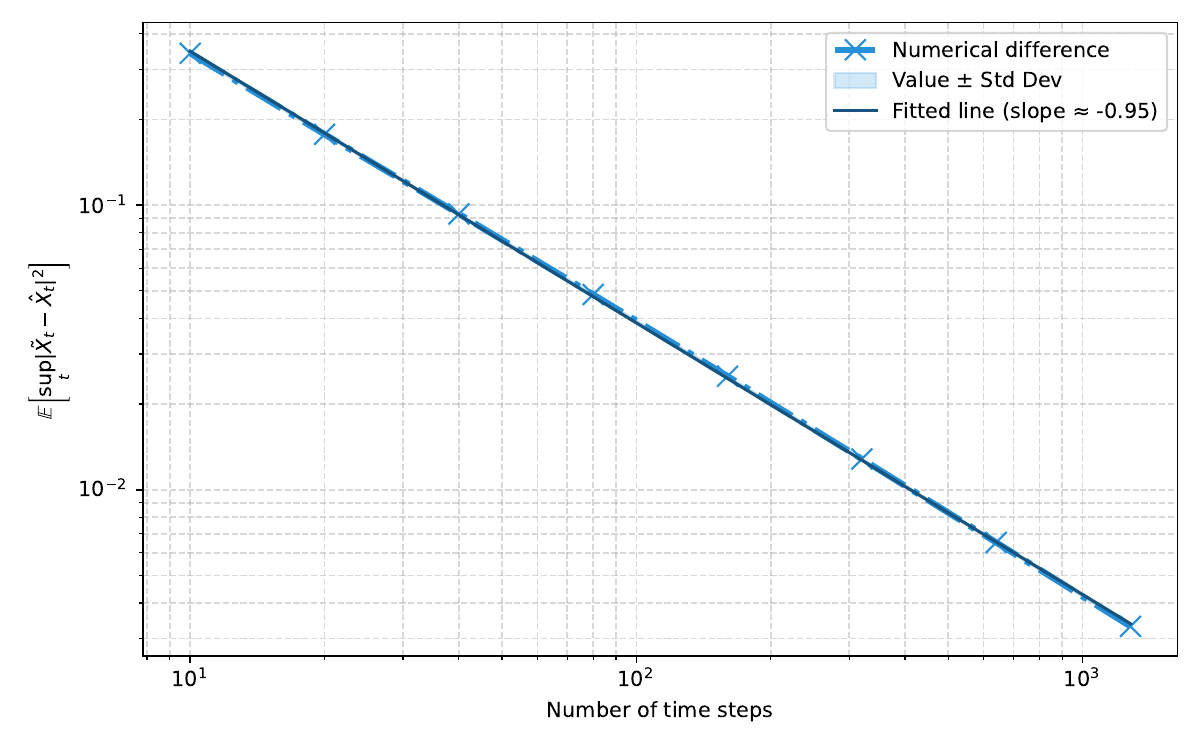}%
    \includegraphics[width=0.48\linewidth]{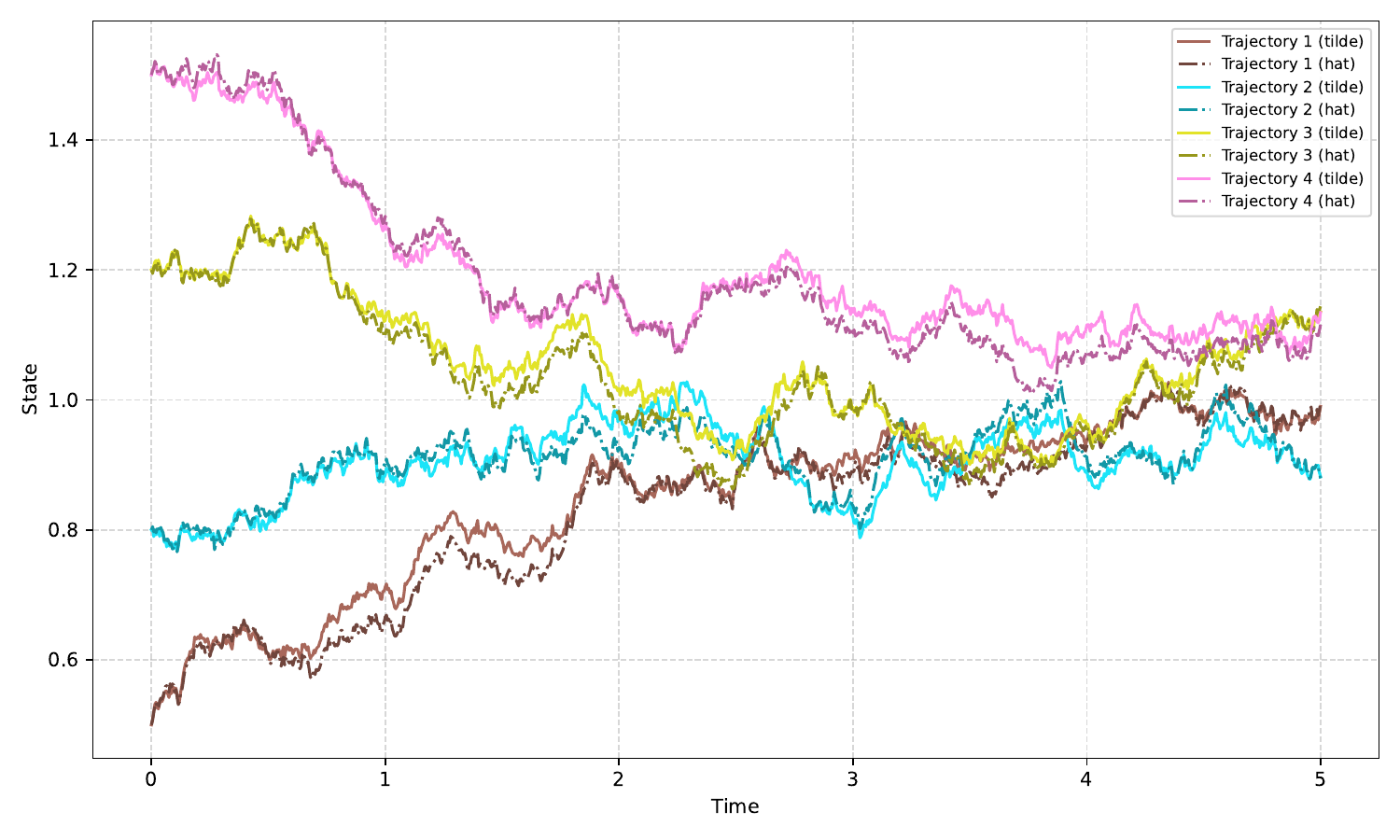}
    \caption{Setting 2. Left: Convergence rate. Right: 4 sample trajectories with $N=1000$.}
    \label{fig:cont-disc-RL-tanh-b-fit-traj}
\end{figure}

{\bf Setting 3. } 
For the state process with controlled volatility, we consider two possible relaxed volatilities for the purpose of the numerics in this setting: a naive approach with 
$$
    \bar\sigma(x)=\int_A\sigma(x,a)\pi(x)(da)
$$
and an alternative approach proposed in~\cite{Wang_Z_Z_RL} with
$$
    \bar\sigma(x)=\sqrt{\int_A\sigma(x,a)\sigma^*(x,a)\,\pi(x)(da)}.
$$
These two approaches are discussed in more detail below and correspond respectively to~\eqref{fo:naive_randomization} and~\eqref{fo:last_attempt}.
Figure~\ref{fig:cont-disc-RL-tanh-b-fit-contsigma} shows how the difference in the left-hand side of~\eqref{fo:final} decreases as  $N$ increases, for several values of the constant $c^a_\sigma$, namely $c^a_\sigma \in \{0.01, 0.05, 0.1, 0.2\}$, for the two proposals of randomization of the volatility discussed in the following section. Results for the use of formula \eqref{fo:naive_randomization} are given in the left pane, while results in the right pane were computed with formula \eqref{fo:last_attempt}. Though different, the results are very similar. As $N$ increases, the curve's decay is slower than in the uncontrolled volatility, and the curve is further from from the linear fit. This numerical experiment validates empirically the conjecture that the \emph{linear rate} of convergence in~\eqref{fo:final} obtained in the case of uncontrolled volatility does not hold in the case of controlled volatility, whether it is with the naive approach or the alternative one.

\begin{figure}[H]
    \centering
    \includegraphics[width=0.48\linewidth]{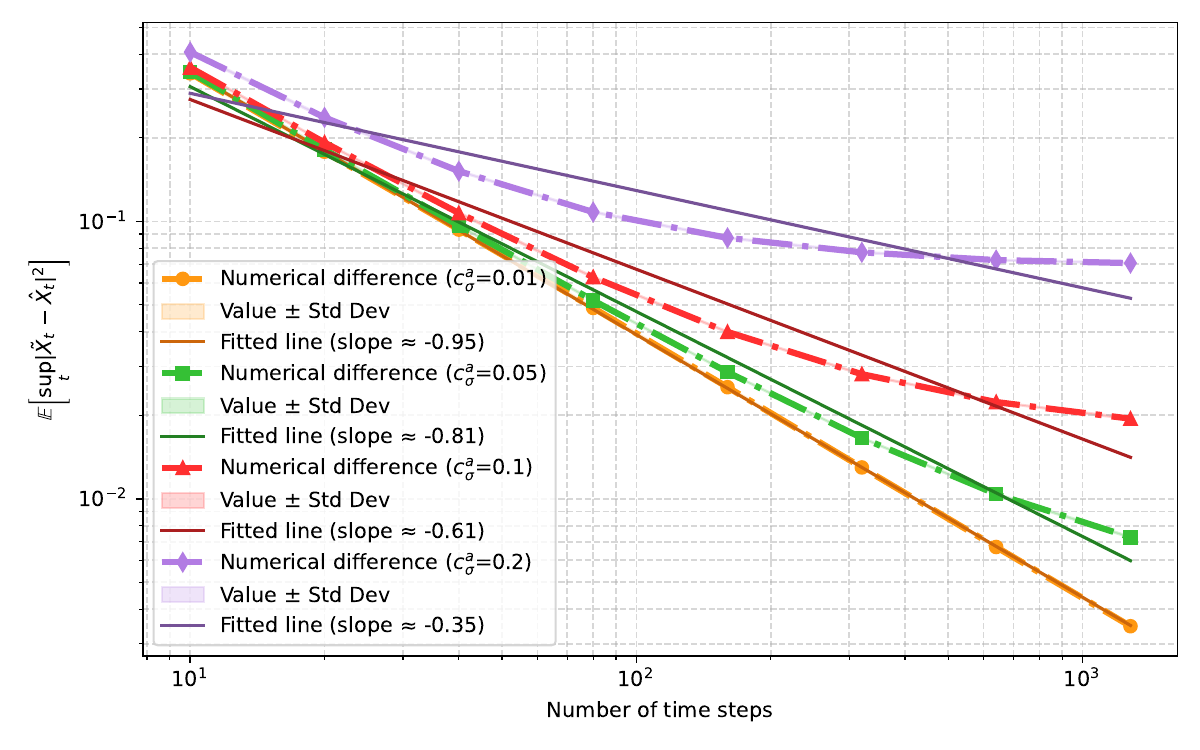}
    \includegraphics[width=0.48\linewidth]{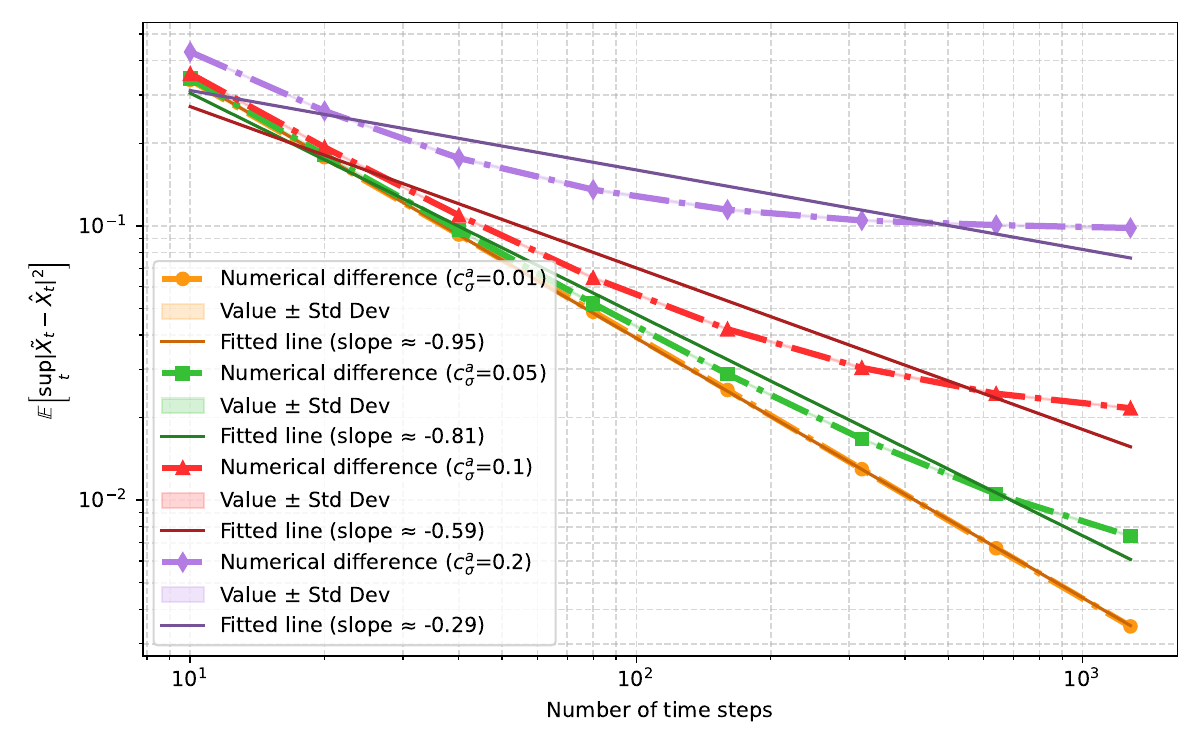}
    \caption{Setting 3. Comparison of convergence rate for several values of $c^a_\sigma$ using formula \eqref{fo:naive_randomization} (left) and formula \eqref{fo:last_attempt} (right).}
    \label{fig:cont-disc-RL-tanh-b-fit-contsigma}
\end{figure}

\section{\textbf{What is so Different when the Volatility is Controlled?}}
\label{se:discussion}

Prompted by the last numerical experiments reported above, we close the paper with a discussion of models based on state dynamics 
$$
	X_t=X_0+\int_0^tb(X_s,\alpha_s)ds+\int_0^t\sigma(X_s,\alpha_s)dW_s
$$
including a controlled volatility term. To alleviate the presentation, we consider, as above, policies that are independent of time. The first relaxation procedure that comes to mind is to use, for a given Markovian relaxed control $\bpi=(\pi(x))_{x\in\RR^d}$, a relaxed state dynamical equation driven by the drift and volatility coefficients:
\begin{equation}
    \label{fo:naive_randomization}
    \bar b(x)=\int_A b(x,a)\pi(x)(da),
    \qquad\text{and}\qquad
    \bar\sigma(x)=\int_A\sigma(x,a)\pi(x)(da).
\end{equation}
Note that $\bar b(x)=\tilde b\bigl(x,\pi(x)\bigr)$ defined in~\eqref{fo:relaxed_coefs} and that by assumption, $\bar\sigma$ is Lipschitz.
Let us assume for a moment that we are willing to enjoy the benefits from the compactification effects of the use of relaxed controls, and accept the fact that the dynamics have been changed: we are now controlling a state and facing costs which are no longer the physical state and costs we were originally supposed to bring to optimality. We could work with dynamics for the new \emph{relaxed} state (still denoted by $X$ for simplicity) given by:
\begin{equation}
\label{fo:new_relaxed}
dX_t=\bar b(X_t)dt + \bar\sigma(X_t)d\bar W_t
\end{equation}
where $(\bar W_t)_{t\ge 0}$ is a Brownian motion defined on a probability space that possibly extends the probability space on which the original Brownian motion  $( W_t)_{t\ge 0}$ is defined.
Given the regular time grid with mesh $T/N$, using the Euler-Maruyama scheme and the notation $\bar \epsilon_{n+1}=\sqrt{T/N}[\bar W_{t_{n+1}}-\bar W_{t_n}]$, we would be led to consider the discrete-time process
\begin{equation}
\label{fo:X_tilde_n}
    \tilde X_{t_n}
    =\tilde X_0+\frac{T}{N}\sum_{m=0}^{n-1}\int_A b\bigl(\tilde X_{t_m},a\bigr)\pi(\tilde X_{t_m})(da)
    +\sum_{m=0}^{n-1}\int_A\sigma\bigl(\tilde X_{t_m},a\bigr)\pi(\tilde X_{t_m})(da)\;\tilde \epsilon_{m+1}.
\end{equation}
Next, in line with our goal to mimic practical RL implementations, the dynamics relevant to the mixed control based approach would be:
\begin{equation}
\label{fo:X_hat_n}
    \hat X_{t_n}
    =\hat X_0+\frac{T}{N}\sum_{m=0}^{n-1} b\bigl(\hat X_{t_m},\hat a_m \bigr)
    +\sum_{m=0}^{n-1}\sigma\bigl(\hat X_{t_m},\hat a_m\bigr)\;\hat\epsilon_{m+1},
\end{equation}
where we use the notation $\hat \epsilon_{n+1}=\sqrt{T/N}[ W_{t_{n+1}}- W_{t_n}]$ for the increments of the (original) physical Brownian motion, and
where each random variable $\hat a_m$ is independent of the past up to (and including) time $t_m$, and $\hat a_m\sim \pi\bigl(\hat X_{t_m}\bigr)$, i.e. $\hat a_m$ is distributed according to the probability distribution $\pi\bigl(\hat X_{t_m}\bigr)$. Accordingly, we could try to prove an upper bound similar to \eqref{fo:final} and complete the proof of the desired convergence. Unfortunately, as illustrated in the numerical experiments reported in {\bf Setting 3} above, the upper bound we can prove is sub-linear, and this is not sufficient to complete the convergence proof. Worse, there are deeper reasons why this randomization is not satisfactory. Indeed, there exist simple examples that show that the value of the optimum may not be the same.

\vskip 2pt
As an effort to resolve the issues with the above naive randomization of the volatility, and  presumably inspired by some of the outcomes of the discussion below, a more sophisticated form of randomization was touted in~\cite{Wang_Z_Z_RL} with the use of  the alternative
\begin{equation}
    \label{fo:last_attempt}
\bar\sigma(x)=\sqrt{\int_A\sigma(x,a)\sigma^*(x,a)\,\pi(x)(da)}.
\end{equation}
Still, the numerical evidence provided by the two plots of Figure \ref{fig:cont-disc-RL-tanh-b-fit-contsigma} shows that in both cases, the rate of convergence is at best sub-linear. 

\vskip 2pt
We believe that~\cite{ElKaroui_et_al} offers an authoritative resolution to the above quandary: rather than randomizing the stochastic differential equation giving the dynamics of the state, one should randomize the corresponding \emph{martingale problem}, hence the infinitesimal generator in the Markovian case. In a sense that we highlight below, this amounts to randomizing the quadratic variation of the diffusion part instead of the volatility. In any case, it turns out that the new optimization problem over relaxed controls has the same value as the original optimization, and compactness properties of the spaces of relaxed controls help prove that these optima are attained. As solutions of martingale problems can also be shown to solve stochastic differential equations, it is tempting to believe that the RL mixed control randomization demonstrated earlier when only the drift is controlled, could still work if one were to use the proper randomized stochastic differential equation. For any $\pi\in\cP(A)$, the diffusion coefficient in the randomization of the martingale problem is given by 
$$
	\frak{a}(x,\pi)=\int \sigma(x,a)\sigma^*(x,a)\, \pi(da),
$$
while the diffusion coefficient derived from the naive randomization suggested above in~\eqref{fo:naive_randomization} would be:
$$
\tilde a(x,\pi)=\tilde\sigma(x,\pi)\tilde\sigma^*(x,\pi)=\Bigl(\int_A\sigma(x,a)\pi(da)\Bigr)\Bigl(\int_A\sigma^*(x,a)\pi(da)\Bigr),
$$
which is obviously different. However, for each $x\in\RR^d$ and $\pi\in\cP(A)$, the matrix $\frak{a}(x,\pi)-\tilde a(x,\pi)$ is symmetric and non-negative definite, so we can construct a symmetric square root $\tilde s(x,\pi)$ satisfying 
$$
	\frak{a}(x,\pi) = \tilde a(x,\pi) + \tilde s(x,\pi)\tilde s^*(x,\pi).
$$
The upshot of this discussion is that, for any Markovian relaxed control $\bpi=(\pi(x))_{x\in\RR^d}$, if $\bX=(X_t)_{t\ge 0}$ is the solution of a relaxed martingale problem with drift and diffusion coefficients 
\begin{equation}
    \label{fo:martingale_randomization}
\tilde b(x,\pi)=\int_A b(x,a)\pi(da),
\qquad\text{and}\qquad
\frak{a}(x,\pi)=\int_A\sigma(x,a)\sigma^*(x,a)\,\pi(da),
\end{equation}
then, the solution of the relaxed martingale problem can be given (in law) by the solution of the stochastic differential equation
\begin{equation}
    \label{fo:last_sde}
d\tilde X^\pi_t=\tilde b\bigl(\tilde X^\pi_t,\pi(X^\pi_t)\bigr)\,dt
+\tilde\sigma\bigl(\tilde X^\pi_t,\pi(X^\pi_t)\bigr)\,dW_t
+\tilde s\bigl(\tilde X^\pi_t,\pi(X^\pi_t)\bigr)\,dB_t
\end{equation}
for two independent Brownian motions $\bW=(W_t)_{t\ge 0}$ and 
$\bB=(B_t)_{t\ge 0}$. Even though the control of the solution of such a stochastic differential equation would lead to the right optimum, it is not possible to use the mixed strategy approach of RL implementations because, not only the time evolution of such a state does not easily relate to the time evolution of the original state, but most importantly, the extra volatility term $\tilde s\bigl(\tilde X^\pi_t,\pi(X^\pi_t)\bigr)$ is not of the form $\int_A s\bigl(\tilde X^\pi_t,a\bigr)\pi(X^\pi_t)(da)$  for a function $\RR^d\times A\ni (x,a) \mapsto s(x,a)$ in general.

\bibliographystyle{apalike}
\bibliography{mfqrl-bib}

\end{document}